\documentclass[12pt,a4paper]{article}
\usepackage{amsmath,amssymb,amsthm}
\usepackage[utf8]{inputenc}
\usepackage[T1, T2A]{fontenc}
\usepackage[russian, english]{babel}
\newtheorem{theorem}{Theorem}
\newtheorem{lemma}{Lemma}
\newtheorem{remark}{Remark}
\usepackage{hyperref}

\author{Nikita V. Gaianov, Anastasia V. Parusnikova}
\title{On finding formal power-logarithmic expansions of solutions to $q$-difference equations}%\thanks{This study was carried out within The National Research University Higher School of Economics Academic Fund Program in 2013-2014, research grant No. 12-01-0030.}}

\date{}

\DeclareMathOperator{\val}{val}

\title{On convergence of formal Dulac series, satisfying an algebraic $q$-difference equation}
\author{Nikita V. Gaianov, Anastasia V. Parusnikova}
\date{\today}

\begin{document}
\maketitle
\begin{abstract}
    An algebraic $q$-difference equation for $|q|>1$ is considered. A sufficient condition is proposed for the convergence of a solution to such an equation in the form of a Dulac series within a sector centered at the origin. An example illustrating the application of this sufficient condition is provided.
%\end{abstract}

\textbf{Keywords:} $q$-difference equation, Dulac series, po\-wer-lo\-ga\-rith\-mic expansion, convergence.

\textbf{ MSC classes:} 	34m25, 34m55.

\end{abstract}

\section{Introduction}
Consider an algebraic  $q$-difference equation
\begin{equation}
    F(x, y, \sigma y, \ldots, \sigma^n y) = 0, \label{eq:QDE}
\end{equation}
where $F = F(x, y_0, \ldots, y_n)$ -- is a polynomial of $n+2$ variables with complex coefficients, $x$  -- is an independent,  $y$ -- is a dependent variable. Both variables are copmplex-valued. Operator  $\sigma$ acts like $(\sigma y)(x)=y(qx)$, where $q \in \mathbb{C}$ и $|q|>1$.

For such equations, methods  of constructing formal solutions in the form of power and generalized power series are known ~\cite{canopoly}, and sufficient conditions for the convergence of the corresponding series in some disk or sector with vertex at the origin are obtained  \cite{gontsov}, \cite{zhang_conv}, \cite{zhang_ex}.

In \cite{us}, we determined for which types of equations a formal solution in the form of a power-logarithmic series (Dulac series) exists. In \cite{Ufa2026}, we transferred the methods and results of Power Geometry~\cite{brumn} to the case of $q$-difference equations, formulated sufficient conditions for the existence of formal solutions to an algebraic $q$-difference equation in the form of power-logarithmic series of a more general form than in the previous work~\cite{us}, and presented a method for obtaining them.

We define Dulac series as a formal series of a form
\begin{equation}
y = \sum_{k\geq 1} p_k(t)x^k ,\label{eq:Dulac}
\end{equation}
%  \in \mathbb{C}[t][[x]]
where $p_k \in \mathbb{C}[t]$,  $t = \log_q x = \frac{\ln x}{\ln q}$, $\arg x \in (-\pi, \pi]$, we consider a fixed main branch of a function $\ln x$.
We define operator $\sigma$ on the linear space of Dulac series as:
$$\sigma \left(\sum_{k\geq1} p_k(t)x^k \right) = \sum_{k\geq1} q^k p_k(t+1) x^k.$$

A Dulac series \eqref{eq:Dulac} is called a formal solution to the equation \eqref{eq:Dulac}, if the result of its substitution to the LHS of the equation \eqref{eq:QDE} is a Dulac series with zero coefficients.

A sufficient condition of convergence in the sector with vertex in zero and angle less than $2\pi$ of formal series of the form \eqref{eq:Dulac} satisfying equation \eqref{eq:QDE} is obtained in this work.
The proof is based on the implicit map theorem for Banach spaces \cite{KolmFomin} and methods from the works  \cite{Malgrange, GonGor19}.

\begin{theorem}[The main theorem]
    Let the equation \eqref{eq:QDE} have a formal solution $y = \varphi$, where
    \begin{equation}
        \varphi = \sum_{k \geq 1} p_k (t) x^k, \label{eq:solution}
    \end{equation}
    and
    $$\frac{\partial F}{\partial y_j}(x, \varphi, \sigma \varphi, \ldots, \sigma^n \varphi) = a_j x^m + b_j(t) x^{m+1} + \ldots, \quad j=0, 1, \ldots, n, $$
    where $a_n \neq 0$, and a number $m \in \mathbb{N}\cup\{0\}$ is the same for all the $j = 0, 1, \ldots, n$. Then there exists $R>0$ such that for any sector $S=\{x: |x|<r, \arg x \in (\varphi_1, \varphi_2)\}$, $r<R$, $\varphi_2-\varphi_1<2\pi$ the series \eqref{eq:solution} is  uniformly convergent on $S$.
    \label{th:main}
\end{theorem}

The results of this work were presented on June 6, 2026, at the conference "XVII Scientific Conference 'Differential Equations and Related Problems of Mathematics'{}", in Kolomna, Russia, and on July 5, 2026, at the conference "International Conference on Differential Equations and Dynamical Systems"{} (DIFF2026), in Suzdal, Russia \cite{Suz2026}.

\section{Reduction of the equation to a special form}
\begin{theorem}
    Under the assumptions of the theorem~\ref{th:main}   there exists $\ell' \in \mathbb{N}$, such that for any $\ell \geq \ell'$
    the substitution $$ y = \sum_{k=1}^\ell p_k(t)x^k + x^\ell u$$ transforms initial equation  \eqref{eq:QDE} to the form
    \begin{equation}
        L(\sigma)u + x\, M(x, t, u, \sigma u, \ldots, \sigma^n u) = 0,\label{eq:specialform}
    \end{equation}
    where
    \begin{enumerate}
        \item $L(\sigma) = \sum_{j = 0}^n a_j q^{j \ell} \sigma^j $ -- is a polynomial of operator $\sigma$,
        \item $\sum_{j=0}^{n-1} \left|\frac{a_j}{a_n} q^{(j - n) \ell} \right|<1,$
        \item $M = M(x, t, y_0, \ldots, y_n)$ is a polynomial of $n+3$ variables.
    \end{enumerate}
\end{theorem}
\begin{proof}
    We define for the series $\varphi = \sum_{k \geq 1} p_k (t) x^k \in \mathbb{C}[t][[x]]$ \emph{an order} $$\val \varphi := \begin{cases}
            +\infty,                        & \varphi = 0;      \\
            \min \{k \mid p_k \neq 0\}, & \varphi \not = 0.
        \end{cases}$$
    For any natural $\ell$ the series $\varphi$ can be presented in a form
    $$\varphi
        =
        \sum_{k=1}^\ell p_k(t) x^k
        +
        x^\ell \sum_{k = 1}^\infty p_{k+\ell}(t)x^k
        =: \varphi_\ell + x^\ell \psi.
    $$

    Let $$ \Phi := (\varphi, \sigma\varphi, \ldots, \sigma^n\varphi), \qquad \Phi_\ell := (\varphi_\ell, \sigma\varphi_\ell, \ldots, \sigma^n\varphi_\ell). $$
    $$ \Phi = \Phi_\ell + x^\ell \Psi, $$ where $$ \Psi := (\psi_0, \ldots, \psi_n), \qquad \psi_j := q^{j\ell}\sigma^j\psi, \qquad j = 0,\ldots,n, $$ т. е. $$ \Psi = (\psi, q^\ell\sigma\psi, q^{2\ell}\sigma^2\psi, \ldots, q^{n\ell}\sigma^n\psi). $$

    From Taylor's formula, we get
    \begin{multline}
        0
        =
        F(x, \Phi)
        =
        F(x, \Phi_\ell + x^\ell \Psi)
        =
        F(x, \Phi_\ell)
        +\\+
        x^\ell \sum_{j=0}^n \frac{\partial F}{\partial y_j}(x, \Phi_\ell)  \psi_j
        +
        \frac{x^{2\ell}}{2} \sum_{i,j=0}^n \frac{\partial^2 F}{\partial y_i \partial y_j}(x, \Phi_\ell)\psi_i \psi_j
        +
        x^{3\ell} R(x, t, \Psi),\label{eq:taylor}
    \end{multline}
    where $R$ -- is a polynomial of $x, t, \psi_0, \ldots, \psi_n$.
    % не содержащий констант, а также слагаемого, зависящего только от $x$, и слагаемых, содержащих только одну из функций $\psi_0, \ldots, \psi_n$, причём в первой степени.

    %
    We choose a natural $\ell$, for what the following conditions are fulfilled
    \begin{enumerate}
        \item $\ell > m$,
        \item
     $\sum_{j=0}^{n-1} \left| \frac{a_j}{a_n}  q^{(j-n) \ell} \right|<1.$
    \end{enumerate}

    Such $\ell$ exists as
    $$\sum_{j=0}^{n-1} \left| \frac{a_j}{a_n}  q^{(j-n) \ell} \right| \leq \sum_{j=0}^{n-1} \left| \frac{a_j}{a_n} \right| |q|^{-\ell}, $$
and $|q|^{-\ell} \to 0$ for $\ell \to +\infty$.

   By Taylor's formula, we obtain
    $$\frac{\partial F}{\partial y_j}(x, \Phi)
        -
        \frac{\partial F}{\partial y_j}(x, \Phi_\ell)
        =
        x^\ell \sum_{i=0}^n\frac{\partial^2 F}{\partial y_i \partial y_j} (x, \Phi_\ell)\psi_i
        +
        \ldots,$$
    moreover $\val \psi_i \geq 1, i=0, \ldots, n$.
    The order of the LHS is not less than $\ell+1 > m$, consequently, both summands on the LHS have the same initial term $a_j x^m$, so,
    \[\frac{\partial F}{\partial y_j} (x, \Phi_\ell)
        =
        a_j x^m + \ldots
    \]
    From the expansion~\eqref{eq:taylor} we obtain that
    $$\val F(x, \Phi_\ell) \geq m+\ell+1.$$
     Dividing the expansion \eqref{eq:taylor} by $x^{m+\ell}$ and isolating the terms linear in $\psi_0,$ $\psi_1, \ldots, \psi_n$, we obtain that the indicated substitution indeed reduces equation \eqref{eq:QDE} to the form \eqref{eq:specialform}.
\end{proof}

The following lemma holds \cite{us}:
\begin{lemma}\label{lemma:1}
    If for $\forall s\in \mathbb{N}\cup\{0\}\;L(q^s) \neq 0 $, then the equation \eqref{eq:specialform} has a unique solution in the form of the Dulac series.
\end{lemma}
\begin{remark}
    In the article \cite{us} the summation in the Dulac series started with zero degree of $x$.
\end{remark}

\begin{lemma} 
    The series $\psi$ is a unique Dulac series, satisfying the equation~\eqref{eq:specialform}.
\end{lemma}
\begin{proof}
    Let us show that $L(q^s) \neq 0$ for all $s\in \mathbb{N}\cup\{0\}$.
    Assume the contrary: let for some $s$ is valid $L(q^s)=0$, i.e.
    $$\sum_{j=0}^n a_j q^{j (\ell+s)} = 0.$$ 
    We divide this equality by $a_n q^{n(\ell+s)}$ and let us move the last term in the sum to the RHS, we obtain
    $$\sum_{j=0}^{n-1} \frac{a_j q^{j(\ell+s)}}{a_n q^{n(\ell+s)}}= -1.$$
    But
    $$
    \left|
    \sum_{j=0}^{n-1} 
    \frac{a_j q^{j(\ell+s)}}{a_n q^{n(\ell+s)}}
    \right| 
    \leq 
    \sum_{j=0}^{n-1}
    \left|
    \frac{a_j}{a_n } q^{(\ell+s)(j-n)} 
    \right| 
    \leq 
    \sum_{j=0}^{n-1} 
    \left|
    \frac{a_j}{a_n}q^{\ell (j-n)} 
    \right| 
    < 
    1, $$
    which contradicts the previous equality. Consequently, $L(q^s)\neq0$, and, by Lemma \ref{lemma:1}, the series $\psi$ is the unique solution of the equation \eqref{eq:specialform}.
\end{proof}
Let us now define $P_{k} := p_{k+\ell}$.
Then for the equation \eqref{eq:specialform} we have a solution $\psi = \sum_{k\geq 1} P_k(t)x^k$.

\begin{theorem}
    $\deg P_k \leq k C,\; \forall k \in \mathbb{N},$ where $C = \deg_t M$ -- the degree of the polynomial $M$ with respect to the variable $t$.
\end{theorem}
We emphasize that the degree of the zero polynomial is assumed to be equal to  $-\infty$.
\begin{proof}
    We prove the statement by mathematical induction on $k$. We substitute the series $\sum_{k=1}^{\infty}P_k(t)x^k$ to the equation \eqref{eq:specialform}
    and equate the coefficients of like powers of $x$. Consider the coefficient of $x^1$. In this case, the monomials from $M$ must not
    depend on $x$, so we obtain the equation
    $$
        L(q\sigma)P_1(t)=-M(0,t,0,\ldots,0).
    $$
    It has a solution beeing a polynomial of the degree
    $
        \deg M(0,t,0,\ldots,0)\leq C.
    $
    Now assume that the assertion is valid for
$k=1,2,\ldots,h-1$.
     We obtain the following equation for $P_h$:
    $$
    L(q^h \sigma) P_h(t) =  R_h(t) ,
    $$
    where $R_h$ is a linear combination of monomials of the form
    $$
    t^\nu  
    (P_{k_1}(t)\ldots P_{k_{s_0}}(t))   
    (P_{l_1}(t+1)\ldots P_{l_{s_1}}(t+1)) 
    \ldots  
    (P_{m_1}(t+n)\ldots 
    P_{m_{s_n}}(t+n)),$$
    and, moreover, $$\sum k_i + \sum l_i + \ldots + \sum m_i \leq h-1.$$
    By the inductive hypothesis
    \[
    \deg P_k \le kC,\; k = 1, \ldots, h-1,
    \]
    so the degree of each monomial is at most  $\nu + C(\sum k_i + \sum l_i + \ldots + \sum m_i) \leq \nu + C(h-1) \leq C h$, consequently, $\deg R_h \leq C h.$ From $L(q^h) \neq 0$ it follows that $\deg P_h = \deg R_h \leq C h.$
\end{proof}

\section{Proof of the main theorem for the equation of the special form}
For any  $\varepsilon > 0$ we can rewrite the solution in the form
\begin{equation}
    \psi = \sum_{k \geq 1}  \tilde{P}_k(\varepsilon \log_{q}x)x^k,
    \label{eq:solution3}
\end{equation}
where $\tilde P_k (t) = P_k\left(\frac{t}{\varepsilon}\right).$

We choose  $\varepsilon$ such that $(1+\varepsilon)^C < |q|$.
In what follows, we will use the notation $\tilde{t} := \varepsilon \log_{q} x $.
For the linear space of polynomials in $\tilde t$ we define norm $\|\cdot\|: \mathbb{C}[\tilde{t}]\to \mathbb{R}_+$ as the sum of the absolute values of the coefficients of the polynomial. We also introduce a family of operators parametrized by $h \in \mathbb{C}: (T_h P)(\tilde t) = P(\tilde t+h)$. It is obvious that $\deg T_h  P = \deg  P$.

\begin{lemma}
    The operators $T_h$ have the following properties:
    \begin{enumerate}
        \item
              $T_{h}^{-1} = T_{-h}$

        \item
              $\|T_h  P\| \leq (1+ |h|)^{\deg P} \|P\|$,

        \item
              $\|T_h  P\| \geq (1+ |h|)^{-\deg P} \|P\|$.
    \end{enumerate}
\end{lemma}
\begin{proof}
    The first assertion follows from the definition. We will prove the second one.
    Let
        \[
            P(\tilde t)=\sum_{k=0}^N c_k \tilde t^k,\qquad
            \|P\|=\sum_{k=0}^N |c_k|.
        \]
       Then
        \begin{multline}
            P(\tilde t+h)=
            \sum_{k=0}^N c_k (\tilde t+h)^k
            =
            \sum_{k=0}^N c_k\sum_{j=0}^k \binom{k}{j} h^{k-j} \tilde t^j = \sum_{j=0}^N \sum_{k=j}^N  \binom{k}{j} c_k h^{k-j} \tilde t^j .
        \end{multline}

       Consequently,
        \[
            \|T_h P\|
            =\sum_{j=0}^N \left|\sum_{k=j}^N c_k  \binom{k}{j} h^{k-j} \right|
            \le \sum_{j=0}^N\sum_{k=j}^N |c_k|\binom{k}{j}|h|^{k-j}
            =\sum_{k=0}^N |c_k|\sum_{j=0}^k \binom{k}{j}|h|^{k-j}.
        \]
       As
        \[
            \sum_{j=0}^k \binom{k}{j}|h|^{k-j}=(1+|h|)^k,
        \]
        we obtain, that
        \[
            \|T_h P\|
            \le \sum_{k=0}^N |c_k|(1+|h|)^k
            \le (1+|h|)^N \sum_{k=0}^N |c_k|
            =(1+|h|)^{N}\|P\|.
        \]

We prove the third statement by rewriting the identity operator as the product of $T_{-h}$ and $T_{h}$:

        $$\|P\| = \|T_{-h} T_h P\|\leq (1+|h|)^{\deg P} \|T_h P\|.$$
\end{proof}

The action of the operator $\sigma$ on the Dulac series
can be rewritten in a form
$$\sigma \left( \sum_{k \geq 1}  \tilde{P}_k(\tilde t)x^k \right) = \sum_{k \geq 1}  (T_\varepsilon \tilde{P}_k)(\tilde t) q^k x^k.$$

    Consider a family of vector spaces
    $$
    H_j = 
    \left\{
    \psi=\sum_{k \geq 1} P_k(\tilde t) x^k 
    \in 
    \mathbb{C}[\tilde t][[x]]  
    \;\middle|\; 
    \deg P_k \leq C k, \;
    \sum_{k \geq 1}  \|T_\varepsilon^j P_k\|\, |q|^{j k} 
    < 
    \infty
    \right\}.
    $$
    We define norms in the vector spaces $H_j$ as
    $$\|\varphi\|_j =  
    \sum_{k \geq 1} \|T_\varepsilon^j P_k\| \, |q|^{j k}.
    $$

In what follows, we will also omit the argument of the coefficient polynomials $P_k$ whenever the standard argument $\tilde t$ is assumed.
    \begin{theorem}
        The vector spaces $H_j$ are Banach spaces.
    \end{theorem}
    \begin{proof}

   Consider a fundamental sequence
    $y_n = \sum_{k \geq 1}P_{nk} x^k \in H_j$, that is
    $\forall \varepsilon_1 > 0\; \exists N \in \mathbb{N}:\;\forall n > N,\; \forall m \in \mathbb{N}$
    \begin{equation}
    \label{fund}
    \sum_{k\geq 1} \|T^j_\varepsilon (P_{nk} - P_{n+m, k})\| |q|^{jk} < \varepsilon_1.
    \end{equation}
    Since all summands on the left-hand side of expression \eqref{fund} are nonnegative, then for any natural $k$
    $$\|T^j_\varepsilon (P_{nk} - P_{n+m, k})\| |q|^{jk} < \varepsilon_1.$$
   Also,
    $$\|T^j_\varepsilon (P_{nk} - P_{n+m, k})\| \geq (1+\varepsilon)^{-Ckj} \|(P_{nk} - P_{n+m, k})\|.$$
    Consequently, for every  $k = k_0 \in \mathbb{N}$ the sequence of polynomials $P_{n k_0}$ is fundamental in a vector space of polynomials of degree not greater than $C k_0$. As the last mentioned vector space is a Banach space as a finite-dimensional one, this sequence has a limit $P_{k_0} := \lim\limits_{n\to \infty} P_{n k_0}$,  where $\deg P_{k_0} \leq C k_0$.
    
    We will prove, that $y_0 = \sum_{k\geq 1} P_k x^k \in H_j$ and $y_n \to y_0$ в $H_j$.
    A fundamental sequence is bounded, i.e.
    $\exists C_1 >0:\; \forall M, n \in \mathbb{N}$  $\sum\limits_{k=1}^M \|T^j_\varepsilon P_{nk}\| |q|^{jk} < C_1.$ 
Passing to the limit as $n \to \infty$,  for each fixed $M$, we obtain that
    $\exists C_1 >0:\; \forall M\in\mathbb{N}\; \sum_{k=1}^M \|T^j_\varepsilon P_{k}\| |q|^{jk} \leq C_1.$
This implies that $\sum\limits_{k \geq 1} \|T^j_\varepsilon P_{k}\| |q|^{jk} \leq C_1 < \infty$, which means that  $y_0 \in H_j$.

Let us now show that $y_n \to y_0$ in $H_j$. From fundamentalness it follows that
    $\forall \varepsilon_1 > 0\; \exists N \in \mathbb{N}:\; \forall M\in\mathbb{N}\;\forall n > N,\; \forall m \in \mathbb{N}$
    $$\sum_{k = 1}^M \|T^j_\varepsilon (P_{nk} - P_{n+m, k})\| |q|^{jk} < \varepsilon_1/2.$$
   Then for $m\to\infty$ we have
    $\forall \varepsilon_1 > 0\; \exists N \in \mathbb{N}:\; \forall M\in\mathbb{N}\;\forall n > N$
    $$\sum_{k = 1}^M \|T^j_\varepsilon (P_{nk} - P_k)\| |q|^{jk} \leq \varepsilon_1/2.$$
    Passing to the limit as $M\to \infty$ and taking into account that $\varepsilon_1 > \varepsilon_1/2$, we obtain $$\|y_n-y_0\|_{j}=\sum_{k \geq 1} \|T^j_\varepsilon (P_{nk} - P_k)\| |q|^{jk} < \varepsilon_1,$$
    and this is exaclty a definition of $y_n \to y_0$ в $H_j$.
\end{proof}
    
    \begin{lemma} \label{pred}
    The spaces $H_j$ have the following properties:
    \begin{enumerate}
        \item  Operator $\sigma$ is an isometry from $H_{j+1}$ to $H_j$, i.e. if $\varphi \in H_{j+1}$, then $\sigma \varphi \in H_{j}$ and $\| \sigma \varphi \|_j = \|\varphi\|_{j+1}$.
    
        \item The embedding $H_{j+1}\subset H_{j}$ holds, and for all $\varphi \in H_{j+1}$ it is true that
 $\|\varphi\|_{j} \leq \|\varphi\|_{j+1} $.
    
        \item If $\varphi\in H_{j_1}$ and $\psi\in H_{j_2}$, then
              $\varphi\psi\in H_0$a and
              \[
                  \|\varphi\psi\|_0 \le \|\varphi\|_0\,\|\psi\|_0 .
              \]
    \end{enumerate}
\end{lemma}
\begin{proof}
    We will prove each property separately.
    \begin{enumerate}
        \item
              Let $\varphi = \sum_{k\geq1} P_k x^k \in H_{j+1}$, i.e. $\deg P_k \leq Ck$,
              $$\|\varphi\|_{j+1} = \sum_{k \geq 1} \|T_\varepsilon^{j+1} P_k\| |q|^{(j+1) k} < \infty,$$
              then
              $$\sigma \varphi = \sum_{k\geq1} q^k (T_\varepsilon P_k) x^k,$$
              $$ \deg q^k T_\varepsilon P_k \leq Ck,$$
              $$\| \sigma \varphi \|_{j}= \sum_{k \geq 1} \| q^{k} T_\varepsilon^{j} T_\varepsilon  P_k \| |q|^{jk} = \sum_{k \geq 1} \|T_\varepsilon^{j+1} P_k \| |q|^{(j+1) k} = \|\varphi \|_{j+1} < \infty,$$
this means, that $\sigma \varphi \in H_{j}$.

        \item
               Let $\varphi \in H_{j+1}$, then
              \begin{align*}
                  \| \varphi \|_j = \sum_{k\geq 1} \| T_\varepsilon^j P_k \| |q|^{jk} = \sum_{k\geq 1} \| T_{-\varepsilon} T_\varepsilon^{j+1} P_k \| |q|^{(j+1)k} |q|^{-k} \leq \\ \leq
                  \sum_{k \geq 1}
                  \left(
                  \frac{(1+\varepsilon)^C}{|q|}
                  \right)^k \|T_\varepsilon^{j+1} P_k\| |q|^{(j+1)k}
                  \leq
                  \frac{(1+\varepsilon)^C}{|q|} \|\varphi\|_{j+1}  \leq \|\varphi\|_{j+1} < \infty,
              \end{align*}
              this means, that $\varphi \in H_j$.
        \item
              Let $\varphi = \sum_k P_k x^k \in H_{j_1}, \psi = \sum_l Q_l x^l \in H_{j_2},$ it follows that
              $$\varphi \psi = \sum_{k\geq 1} \sum_{l = 1}^{k-1} P_l Q_{k-l} x^k.
              $$
              $$\deg  P_l Q_{k-l} \leq C l + C(k-l) = C k.$$

              $$
                  \|\varphi \psi\|_{0} =
                  \sum_{k\geq 1} \left\| \sum_{l=1}^{k-1} P_l Q_{k-l} \right\|
                  \leq
                  \sum_{k\geq 1} \|P_k\| \sum_{k\geq 1} \|Q_k\| = \|\varphi\|_0 \|\psi\|_0  < \infty.
              $$
    \end{enumerate}

\end{proof}
\begin{theorem}
Operator $L(\sigma) : H_{n} \to H_0$ is continuous and has a continuous inverse.
\end{theorem}
\begin{proof}
        Operator  $L(\sigma)$ is continuous as the operators $\sigma^j : H_{n} \to H_0, \; j=0, 1, \ldots n$ are  continuous. Operator $\sigma^j$ is continuousfollows from the properties in lemma \ref{pred}.
        Let us show that $L(\sigma)$ is invertible. Consider the operator 
            $$\sigma: \mathbb{C}[\tilde t][[x]] \to \mathbb{C}[\tilde t][[x]],$$ it has an inverse operatod, defined by the formula
            $$ \sigma^{-1} \left(\sum_{k \geq 1} P_k x^k \right) = \sum_{k \geq 1} q^{-k} (T_{-\varepsilon} P_k) x^k. $$
            The restriction of the operator $\sigma^{-1}$ to $H_0$, which we will denote by $\tilde\sigma^{-1}$, has image in $H_0$, and $\|\tilde\sigma^{-1}\|_{B(H_0)} \leq \frac{(1+\varepsilon)^C}{|q|} $. Indeed, if $\varphi = \sum\limits_{k\geq 1}P_k x^k \in H_0,$ then
                \begin{gather*}
                    \|\tilde \sigma^{-1} \sum_{k\geq 1}P_k x^k \|_0 = \sum_{k \geq 1} \|T_{-\varepsilon} P_k\| |q|^{-k} \leq \\ \leq \sum_{k \geq 1} \left(\frac{(1 + \varepsilon)^C}{|q|}\right)^k \|P_k\| \leq \frac{(1+\varepsilon)^C}{|q|} \|\varphi\|_0.
                \end{gather*}
            
        Let us consider an equation
        \begin{equation}
            L(\sigma)y = b, b \in H_{0}. \label{eq:lineq}
        \end{equation}
        First, we show that there exists a unique solution in the form of a formal Dulac series. We substitute the solution in the form $y = \sum\limits_{k\geq 1} P_k(\tilde t)x^k$, $b = \sum\limits_{k\geq 1} B_k(\tilde t) x^k.$ Equating the coefficients of equal powers of the variable $x$, we obtain an inhomogeneous system of linear difference equations with constant coefficients and with right-hand side in the form of a polynomial:    $$L(q^k T_\varepsilon) P_k = B_k.$$
            As $L(q^k) \neq 0$ for $k\in \mathbb{N}$, then, using the theory of linear difference equations, each of the equations has a unique solution in a form of polynomial with degree equal to the degree of the corresponding RHS, i.e. $\deg P_k \leq Ck$.
            
        It remains to show that $y$ lies in $H_n$. To do this, we pass from the equation to a system by the substitution $\sigma^j y = \mathbf{y}_j$.

        \begin{equation}\label{eq:syst1}
            \begin{cases}
                \sigma \mathbf{y}_0 = \mathbf{y}_1,         \\
                \ldots                                      \\
                \sigma \mathbf{y}_{n-2} = \mathbf{y}_{n-1}, \\
                \sigma \mathbf{y}_{n-1} = -\frac{1}{a_n}\sum_{j=0}^{n-1} a_j q^{(j-n) \ell} \mathbf{y}_j + \frac{b}{q^{n \ell } a_n}.
            \end{cases}
        \end{equation}

        We apply the operator $\sigma^{-1}$ to both sides, after which we will consider only those solutions where $\mathbf{y}_j \in H_0, j = 0, 1, \ldots, n-1$; then $\sigma^{-1}$ can be restricted to $H_0$ and rewritten in the form $\tilde \sigma^{-1}$:

        $$\begin{cases}
                \mathbf{y}_0 = \tilde \sigma^{-1}\mathbf{y}_1,         \\
                \ldots                                                 \\
                \mathbf{y}_{n-2} = \tilde \sigma^{-1}\mathbf{y}_{n-1}, \\
                \mathbf{y}_{n-1} = -\frac{1}{a_n}\sum_{j=0}^{n-1} a_j q^{(j-n) \ell} \tilde \sigma^{-1} \mathbf{y}_j + \frac{\tilde \sigma^{-1} b}{q^{n \ell } a_n}.
            \end{cases}$$

        Now we define a Banach space $\mathcal{H}_0 = H_0^n = H_0 \times H_0 \times \ldots \times H_0$, with norm
        $$
\|(\mathbf y_0,\dots,\mathbf y_{n-1})\|_{\mathcal H_0}
:=
\max_{0\le j\le n-1}\|\mathbf y_j\|_0.
$$

        Our problem is reduced to finding a fixed point of the mapping
        \begin{align*}
            (\mathbf{y}_0, \ldots, \mathbf{y}_{n-1}) = \mathbf{y} \mapsto (\tilde\sigma^{-1} \mathbf{y}_1 ,\ldots, \tilde\sigma^{-1} \mathbf{y}_{n-1}, -\frac{1}{a_n}\sum_{j=0}^{n-1} a_j q^{(j-n) \ell} \tilde \sigma^{-1} \mathbf{y}_j) + \\ + (0, \ldots, 0, + \frac{\tilde \sigma^{-1} b}{q^{n \ell } a_n}) =  \tilde{L}\mathbf{y} + \frac{\tilde \sigma^{-1} b}{q^{n \ell }a_n} \mathbf{e}_n.
        \end{align*}
        Let us show that this mapping is contractive, which will be equivalent to the linear operator $\tilde L$ being bounded and having norm less than one.

        $$\|\tilde{L}\|_{B(\mathcal{H}_0)} = $$ $$= \sup_{\|\mathbf{y}\|\leq 1} \max \left(\|\tilde\sigma^{-1} \mathbf{y}_1\|_0 , \ldots, \|\tilde\sigma^{-1} \mathbf{y}_{n-1}\|_0, \left\|-\frac{1}{a_n}\sum_{j=0}^{n-1} a_j q^{(j - n)\ell} \tilde\sigma^{-1} \mathbf{y}_j \right\|_0 \right).$$
        We obtain that $\|\tilde\sigma^{-1} \mathbf{y}_j\|_0 \leq \frac{(1+\varepsilon)^C}{|q|} \|\mathbf{y}_j\|_0 \leq \frac{(1+\varepsilon)^C}{|q|} \|\mathbf{y}\|_{\mathcal{H}_0}$,

        \begin{gather*}
            \left\|-\frac{1}{a_n}\sum_{j=0}^{n-1} a_j q^{(j-n)\ell} \tilde\sigma^{-1} \mathbf{y}_j \right\|_0 \leq \sum_{j=0}^{n-1} |q|^{(j-n)\ell} \frac{|a_j|}{|a_n|} \|\tilde\sigma^{-1} \mathbf{y}_j\|_0 \leq\\ \leq \sum_{j=0}^{n-1} |q|^{(j-n)\ell}\frac{|a_j|}{|a_n|}\frac{(1+\varepsilon)^C}{|q|} \|\mathbf{y}\|_{\mathcal{H}_0} \leq \frac{(1+\varepsilon)^C}{|q|} \|\mathbf{y}\|_{\mathcal{H}_0}.
        \end{gather*}

       From this we obtain that the mapping is contractive, $|\tilde{L}| \leq \frac{(1+\varepsilon)^C}{|q|}< 1$, consequently, by the contraction mapping theorem, there exists a unique fixed point $\mathbf{y}$, which belongs to $\mathcal{H}_0$, hence the unique solution of equation \eqref{eq:lineq} is exactly
 $y = \mathbf{y}_0$. From the last equality of system \eqref{eq:syst1} it follows that   $\sigma^n y = -\frac{1}{a_n}\sum_{j=0}^{n-1} q^{(j-n)\ell} a_j  \mathbf{y}_j + \frac{b}{a_n q^{n\ell}} \in H_0$, so, $y \in H_n$.
    \end{proof}

Now consider the equation \eqref{eq:specialform}. If $y = \sum_{k \geq 1} P_k(\log_q x) x^k$ is a unique aolution in a form of the Dulac series for the equation \eqref{eq:specialform}, then the series $\tilde y = \sum_{k \geq 1} P_k(\log_q x) (\lambda x)^k$ is a unique solution in a form of Dulac series for the equation
$$
L(\sigma) \tilde y + \lambda x M(\lambda x, \log_q x, \tilde y, 
\ldots, \sigma^n \tilde y) = 0.
$$
We can rewrite this, using the notation $\tilde t = \varepsilon \log_q x$ in the following form:
$$
L(\sigma) \tilde y + \lambda x \tilde M(\lambda x, \tilde t, \tilde y, 
\ldots, \sigma^n \tilde y) = 0.
$$

Consider the mapping
$$
A(\lambda, \psi) = L(\sigma)\psi + \lambda x \tilde M (\lambda x, \tilde t, \psi, 
\ldots, \sigma^n \psi).
$$

We apply the following implicit mapping theorem from \cite{KolmFomin} to it:

\begin{theorem}
Let $X, Y, Z$ be Banach spaces, $U$ -- a neibourhood of the point $(x_0, y_0) \in  X \times Y$ and $F$ -- a mapping $U$ to $Z$,
having the following properties:
\begin{enumerate}
    \item $F$ is continuous at the point $(x_0, y_0)$;
    \item $F(x_0, y_0) = 0$;
    \item the partial derivative $F'_y(x,y)$ exists in $U$ and is continuous at the point $(x_0, y_0)$, and the operator $F'_y(x_0, y_0)$ has a bounded inverse.
\end{enumerate}
Then the equation $F(x,y)=0$ is solvable for $y$ in some neighborhood of the point $(x_0, y_0)$.
\end{theorem}

We apply this theorem, taking $X = \mathbb{R},\; Y = H_n,\; Z = H_0$. All these three spaces are Banach.

Let us prove that the conditions are satisfied. Let us show that $\mathrm{Im}\, A \subset H_0$. It was shown earlier that $L(\sigma)$ maps $H_n$ into $H_0$. The expression
$$
\lambda x \tilde M  (\lambda x, \tilde t, \psi, 
\ldots, \sigma^n \psi)
$$
is a polynomial in variables
$$
\lambda x,\; \lambda x\tilde t,\; \ldots,\; \lambda x\tilde t^C,\; \psi,\; \sigma \psi,\; \ldots,\; \sigma^n \psi.
$$
Each of these expressions belongs to $H_0$, Consequently, each monomial — and hence the entire expression — belongs to $H_0$.

We show that the mapping $A$ is differentiable. Consider
$$
A(\lambda, \psi+ \eta)-A(\lambda, \psi)= L(\sigma)\eta + \lambda x \sum_{j=0}^n \frac{\partial \tilde M}{\partial y_j}(
        \lambda x, \tilde t, \psi, \ldots, \sigma^n \psi) \sigma^j \eta + R,
$$

where $R = \sum_{i_1, i_2} C_{i_1, i_2} (\sigma^{i_1}\eta)(\sigma^{i_2} \eta)+ \ldots + \sum_{i_1, \ldots, i_N}C_{i_1, \ldots, i_N} (\sigma^{i_1}\eta)\ldots (\sigma^{i_N}\eta)$, where $C_{i_1, \ldots, i_k}$ are the expressions, which do not depend on $\eta$ and its $q$-difference derivatives.

%---------------------------------
It is clear that as $\eta \to 0$
$$
\|\sigma^j \eta \|_0 = O(\|\eta\|_n),\, j= 0, 1, \ldots, n,
$$
and
$$
\|(\sigma^{i_1} \eta)\ldots(\sigma^{i_k}\eta)\|_0 
= O(\|\eta\|_n^k) = o(\|\eta\|_n), \quad k \geq 2.
$$
Therefore,
$$
\|R\|_0 = o(\|\eta\|_n).
$$

We have
$$
A'_\psi (\lambda, \psi)\eta = L(\sigma)\eta + \lambda x \sum_{j=0}^n \frac{\partial \tilde M}{\partial y_j}(
        \lambda x, \tilde t, \psi, \ldots, \sigma^n \psi) \sigma^j \eta.
$$

Set $(\lambda, \psi)=(0, 0)$, then
$$
A'_\psi (0, 0) = L(\sigma).
$$
By similar reasoning, we conclude, that expression  $A'_\psi(\lambda, \psi) - A'_\psi(0, 0)$ is equal to $R \eta = \sum_{j=0}^n C_j\sigma^j \eta $, where $C_j$ -- is a polynomial in  $\lambda x, \lambda x \tilde t,$ $ \ldots, \lambda x \tilde t^C,$ $ \psi_0, \ldots, \psi_n$, with no constant term.
It follows that 
$$\|R\|_{B(H_n, H_0) } \leq \sum_{j=0}^n \|C_j\|_0 \|\sigma^j\|_{\mathcal{B}(H_n,H_0) }.$$
Each norm $\|C_j\|_0$ aims to zero as  $(\lambda, \psi) \to (0,0)$, thus, 
$$\| A'_\psi(\lambda, \psi) - A'_\psi(0, 0) \|_{B(H_n,H_0) } \to 0,$$ i.e. derivative is convergent at zero.

Since $A'_\psi(0,0)=L(\sigma)$ is a homeomorphism, theorem's assumptions are satisfied.

By the implicit map theorem, there exists $r>0$, that for all $\lambda$ with $|\lambda|<r$, there exists a Dulac series 
$$
\psi_\lambda \in H_n,
$$
for which $A(\lambda,\psi_\lambda)=0$, i.e. $\psi_\lambda$ -- is a solution to the equation, moreover, $\psi_\lambda \in H_n \subset H_0$.

From the uniqueness of the solution in the form of a Dulac series, we obtain that  $\psi_\lambda = \sum_{k=1}^\infty \tilde P_k(\tilde t) \lambda^k x^k$.
We fix $0 <\lambda < r$.
As $\psi_\lambda \in H_0$, we have
$$\sum_{k\geq 1} \lambda^k \|\tilde P_k\| < \infty.$$
Let us now pass on to  the proof of convergence. We choose a sector $S$ of opening less than $2\pi$ and an integer $\rho>C$ such that for all $x\in S$ the following conditions are satisfied:
\begin{enumerate}
    \item $|\varepsilon \log_q x| < |x|^{-1/\rho}$,
    \item $|x|^{1-\frac{C}{\rho}} < |\lambda|.$
\end{enumerate}

Using the first condition, we obtain that for every polynomial $P$ the following estimate is valid:
$$
|P(\varepsilon \log_q(x))|
\leq
\|P\|\, |x|^{-\deg P/\rho}.
$$

Let us estimate the modulus $\psi(x)$ from above:
\begin{gather}
    |\psi(x)|
    =
    \left|
    \sum_{k\ge1} \tilde P_k(\varepsilon \log_q x) x^k
    \right|
    \leq
    \sum_{k\ge1}
    \|\tilde P_k\|\, |x|^{-Ck/\rho}  |x|^k
    \leq
    \\
    \leq
    \sum_{k\geq 1}
    \|\tilde P_k\|\, |x|^{(1-C/\rho)k} 
    < \sum_{k\geq 1} \|\tilde P_k\|\, |\lambda|^k < \infty.
\end{gather}
In the penultimate inequality, we used the second condition on the sector $S$.

\section{Application of the convergence theorem}
Consider the $q$-difference equation
 $$(y + q x - q y - q x y) \sigma y - q x y = 0,\quad q = e.$$
It has a formal solution $\varphi$ in a form of a series
\begin{equation}
\label{ex_series}
    \varphi =  \sum_{k \geq 1} t^{k-1}x^k.
\end{equation}
 
$$F = (y_0 + q x - q y_0 - q x y_0) y_1 - q x y_0.$$
$$F'_{y_0} = (1 - q - qx )y_1 - q x.$$
$$F'_{y_1} = q x + (1 - q - q x) y_0.$$
$$F'_{y_0}|_{\varphi} = - q^2 x + \ldots,\; F'_{y_1}|_{\varphi} = x + \ldots $$
The conditions of the theorem are satisfied, hence the series converges in any sector with center at the origin of sufficiently small radius and opening less than `2$\pi$`.

In this sector the sum of the series is a function $\frac{x}{1-x\ln x}$.

\textbf{Acknowledgments.}  Support from the Basic Research Program of HSE University is gratefully acknowledged (HSE-BR-2025-079).

\vskip 1 cm
\textbf{Affiliations.}\\
HSE University,\\
%Moscow Institute of Electronics and Mathematics,\\
Tallinskaya 34, Moscow, 123458, Russia\\
e-mails: parus-a@mail.ru, gajanovnv@gmail.com

\end{document}